\documentclass{article}

\usepackage{amsmath}
\usepackage{amssymb}
\usepackage{amsthm}
\usepackage{hyperref}
\usepackage{url}
\usepackage{mathbbol}

\newtheorem{theorem}{Theorem}[section]
\newtheorem{lemma}[theorem]{Lemma}

\theoremstyle{definition}

\theoremstyle{remark}
\newtheorem{remark}{Remark}[section]
\newtheorem{conjecture}{Conjecture}

\allowdisplaybreaks

\date{}

\begin{document}

\title{Greedy lattice paths with general weights}

\author{Yinshan Chang\thanks{Address: College of Mathematics, Sichuan University, Chengdu 610065, China; Email: ychang@scu.edu.cn; Supported by National Natural Science Foundation of China \#11701395.}, Anqi Zheng\thanks{Address: College of Mathematics, Sichuan University, Chengdu 610065, China; Email: 253104278@qq.com}}

\date{}

\maketitle

\begin{abstract}
Let $\{X_{v}:v\in\mathbb{Z}^d\}$ be i.i.d. random variables. Let $S(\pi)=\sum_{v\in\pi}X_v$ be the weight of a self-avoiding lattice path $\pi$. Let \[M_n=\max\{S(\pi):\pi\text{ has length }n\text{ and starts from the origin}\}.\]
We are interested in the asymptotics of $M_n$ as $n\to\infty$.

This model is closely related to the first passage percolation when the weights $\{X_v:v\in\mathbb{Z}^d\}$ are non-positive and it is closely related to the last passage percolation when the weights $\{X_v,v\in\mathbb{Z}^d\}$ are non-negative. For general weights, this model could be viewed as an interpolation between first passage models and last passage models. Besides, this model is also closely related to a variant of the position of right-most particles of branching random walks.

Under the two assumptions that $\exists\alpha>0$, $E(X_0^{+})^d(\log^{+}X_0^{+})^{d+\alpha}<+\infty$ and that $E[X_0^{-}]<+\infty$, we prove that there exists a finite real number $M$ such that $M_n/n$ converges to a deterministic constant $M$ in $L^{1}$ as $n$ tends to infinity. And under the stronger assumptions that $\exists\alpha>0$, $E(X_0^{+})^d(\log^{+}X_0^{+})^{d+\alpha}<+\infty$ and that $E[(X_0^{-})^4]<+\infty$, we prove that $M_n/n$ converges to the same constant $M$ almost surely as $n$ tends to infinity.
\end{abstract}

\section{Introduction}
 Let $\{X_{v}:v\in\mathbb{Z}^d\}$ be i.i.d. random variables. We consider self-avoiding paths in a $d$-dimensional lattice defined as follows: A \emph{self-avoiding path} of length $n$ starting from $v$ is a sequence of different vertices $v_1=v,v_2,\ldots,v_n$ such that $v_i$ and $v_{i+1}$ are adjacent on the graph $\mathbb{Z}^d$. For a self-avoiding path $\pi$, its weight is defined by
 \[S(\pi)=\sum_{v\in\pi}X_v.\]
 Define
 \[M_n=\max\left\{S(\pi):\begin{array}{l}\pi\text{ is a self-avoiding path of length }n\\
 \text{starting from the origin}\end{array}\right\}.\]
 If $S(\pi)=M_n$ for a path of length $n$ starting from the origin, we call $\pi$ a \emph{greedy lattice path}.

 In the present paper, we are interested in the asymptotics of $M_n$ as $n\to\infty$. When $\{X_{v}:v\in\mathbb{Z}^d\}$ are i.i.d. \emph{non-negative} random variables, Gandolfi and Kesten have proved in \cite{GandolfiKestenMR1258174} that there exists $M\in[0,+\infty)$ such that
 \[\frac{M_n}{n}\overset{n\to\infty}{\to}M\quad\text{a.s. and in }L^{1}\]
 under the moment condition
 \[\exists \alpha>0,E(X_0^d(\log^{+}X_0)^{d+\alpha})<+\infty.\]
 One motivation for considering general weights instead of non-negative weights is to generalize the results of Gandolfi and Kesten. We show that similar results hold when $X_0$ could possibly take negative values:
 \begin{theorem}\label{thm: linear growth}
  Let $x^{+}=\max(x,0)$, $x^{-}=\max(-x,0)$, $\log^{+}(x)=\max(\log x,0)$. Assume that there exists $\alpha>0$ such that \[E((X_0^{+})^d(\log^{+}X_0^{+})^{d+\alpha})<+\infty\]
  and that $E(X_0^{-})<+\infty$. Then, there exists $M\in(-\infty,+\infty)$ such that
  \[\frac{M_n}{n}\overset{n\to\infty}{\to}M\quad\text{ in }L^{1}.\]
  If we further assume that $E((X_0^{-})^4)<+\infty$, then
  \[\frac{M_n}{n}\overset{n\to\infty}{\to}M\quad\text{ a.s.}.\]
 \end{theorem}
 \begin{remark}
  For $d\geq 2$, when the distribution of $X_0$ is non-degenerate, the limit $M$ in Theorem~\ref{thm: linear growth} is strictly greater than $EX_0$. To show this, one could use the idea of the proof of Theorem~7.4 in \cite{SmytheWiermanMR513421}.
 \end{remark}

 Next, we wish to discuss the other two motivations for considering the generalization from positive weights to general weights. The second motivation for considering general weights is to unify first passage models and last passage models. Indeed, when the weights $X_v$ are non-positive, our model is closely related to first passage percolations. We refer to \cite{AuffingerDamronHansonMR3729447} for an overview of first passage percolation models. View $-X_v$ as the time for passing through the vertex $v$. Let $C(t)$ be the cluster of vertices that could be reached within time $t$ from the origin. Then, if $-M_n>t$, then each self-avoiding path of length $n$ has the passage time greater than $t$. Hence, the distance from the origin to the boundary of $C(t)$ is less than $n$. Or equivalently, the inner radius of $C(t)$ is less than $n$. Similarly, when the weights $X_v$ are non-negative, the behavior of $M_n$ is closely related to last passage percolations. So, our model with general weights generalizes both the first passage models and last passage models. It serves as an interpolation between these two kinds of models. Thus, by looking at the model with general weights, there may exist a chance for discovering the connection between first passage models and last passage models and for better understanding these two kinds of models.

 The third motivation for studying the case of general weights is related to branching random walks. We refer to \cite{ShiMR3444654} for a detailed study of branching random walks. Consider one dimensional branching random walks with deterministic binary branching mechanism. One can interpret $X_v$ as one-step displacement of a particle in branching random walks. Then, the position of the right-most particle of the $n$-th generation is precisely $M_{n+1}-X_0$ on the binary tree, where the origin of $\mathbb{Z}^d$ is replaced by the root of the binary tree. If the branching process is a super-critical Galton-Watson process, then the position of the right-most particle of the $n$-th generation is exactly $M_{n+1}-X_0$ on the Galton-Watson tree. Again, compared with the model on $\mathbb{Z}^d$, the starting point of self-avoiding paths is the root of the Galton-Watson tree instead of the origin of $\mathbb{Z}^d$. In this manner, the position of the right-most particle of a one dimensional branching random walk is closely related with the weight of the greedy path on a random tree. One may make a natural extension and consider similar problems on other graphs instead of trees. We consider the problem on $\mathbb{Z}^d$. The correlation is stronger and the problem is more difficult due to the absence of the tree structure in our opinion. For $M_n$ on the binary tree, it is known that $M_n$ grows asymptotically linearly in $n$ under natural assumptions on the weights, see \cite[Theorem~1.3]{ShiMR3444654} for the speed of left-most particle in general branching random walks. (By considering the mirror symmetry about the origin,  the left-most particle changes to the right-most particle.) It is interesting to see that $M_n$ also grows linearly in $n$ on the graph $\mathbb{Z}^d$ according to Theorem~\ref{thm: linear growth}.

 Next, we briefly present the proof strategy. The general idea of the proof is to truncate the variables $\{X_{v}:v\in\mathbb{Z}^d\}$ and consider the weight $M_n^{\geq -m}$ of the greedy lattice paths associated with the weights $\{X_v\vee(-m):v\in\mathbb{Z}^d\}$. By Theorem~1 in \cite{GandolfiKestenMR1258174}, $M_n^{\geq -m}/n$ converges towards a finite real number $M^{\geq -m}\geq -m$ almost surely and in $L^{1}$ as $n\to\infty$. Finally, we show that $M_n^{\geq -m}$ is a good approximation of $M_n$ and $M_n/n$ converges to $M=\inf_{m}M^{\geq -m}$. This general idea is similar to the study of greedy lattice animals with negative weights in \cite{DemboGandolfiKestenMR1825148}. We give the details of the proof in Section~\ref{sect: proof of main theorem}.

 Our results are not optimal. We present conjectures on the possibly optimal result in Section~\ref{sect: open problems}.

\section{Proof of the linear growth}\label{sect: proof of main theorem}

For an event $A$, let $\mathbb{1}[A]$ be the indicator function of $A$. Fix an order $\preceq$ of $\mathbb{Z}^d$. It induces a lexicographic order on the space of lattice paths of length $n$. For $m>0$, let $\pi^{\geq -m}_{n}$ be the greedy lattice path of length $n$ starting from the origin associated with the weights $\{X_v\vee(-m):v\in\mathbb{Z}^d\}$ such that $\pi^{\geq -m}_{n}$ is the smallest in lexicographic order among all such greedy lattice paths. Recall that $M^{\geq -m}_{n}$ is the weight of $\pi^{\geq -m}_{n}$, and that $M_n$ is the weight of a greedy lattice path of length $n$ starting from the origin associated with the weights $\{X_v:v\in\mathbb{Z}^d\}$. By \cite[Theorem~1]{GandolfiKestenMR1258174}, since there exists $\alpha>0$ such that \[E((X_0+m)\vee 0)^d(\log^{+}((X_0+m)\vee 0))^{d+\alpha})<+\infty,\]
there exists $M^{\geq -m}\in(-\infty,+\infty)$ such that
\begin{equation}\label{eq: truncated cvg}
\frac{M^{\geq -m}_n}{n}\overset{n\to\infty}{\to}M^{\geq -m}\quad\text{a.s. and in }L^{1}.
\end{equation}
Define
\begin{equation}\label{eq: defn M}
 M=\lim_{m\to\infty}M^{\geq -m}.
\end{equation}
Note that $M^{\geq -m}\geq E(X_0\vee (-m))$. Hence, $M\in[E(X_0),+\infty)$. Note that $M_n\leq M^{\geq -m}_{n}$ for any $m>0$. On the other hand, for $m>0$, we have that
\[M_n\geq S(\pi^{\geq -m}_{n})=\sum_{v\in\pi^{\geq -m}_{n}}X_v=M_n^{\geq -m}-\sum_{v\in\pi^{\geq -m}_{n}}(-m-X_v)\mathbb{1}[X_v\leq -m].\]
Hence, we have that
\begin{align}\label{eq: upper bound of |M_n-M|}
 |M_n/n-M|\leq & |M^{\geq -m}_{n}/n-M^{\geq -m}|+|M^{\geq -m}-M|\notag\\
 &+\sum_{v\in\pi^{\geq -m}_{n}}(-m-X_v)\mathbb{1}[X_v\leq -m]/n.
\end{align}
Firstly, we prove the $L^1$ convergence: Given the greedy lattice path $\pi^{\geq -m}_{n}$, the conditional expectation of $\sum_{v\in\pi^{\geq -m}_{n}}(-m-X_v)\mathbb{1}_{X_v\leq -m}$ is equal to \[E(-m-X_0|X_0\leq -m)N_n(m),\]
where $N_n(m)$ is the number of sites on $\pi^{\geq -m}_{n}$ with weight $-m$, i.e.
 \[N_n(m)=\sum_{v\in\pi^{\geq -m}_{n}}\mathbb{1}[X_v\leq -m].\]
Then, we have that
\begin{equation}\label{eq: upper bound of expectation of increament along pimn}
E\left(\sum_{v\in\pi^{\geq -m}_{n}}(-m-X_v)\mathbb{1}[X_v\leq -m]\right)= E(N_n(m))E(-m-X_0|X_0\leq -m).
\end{equation}

We will need the following lemma.
\begin{lemma}\label{lem: key lemma}
 For $k\geq 1$, we have that
 \[E\prod_{j=0}^{k-1}(N_n(m)-j)\leq \prod_{j=0}^{k-1}(n-j)P(X_0\leq -m)^k.\]
\end{lemma}
\begin{remark}
 It suffices to prove Lemma~\ref{lem: key lemma} for $n\geq k$. For $n\leq k-1$, both sides equal to $0$ and the inequality trivially holds.
\end{remark}
We postpone the proof of Lemma~\ref{lem: key lemma} and proceed with the proof of Theorem~\ref{thm: linear growth}. By \eqref{eq: upper bound of |M_n-M|}, \eqref{eq: upper bound of expectation of increament along pimn} and Lemma~\ref{lem: key lemma} with $k=1$, we have that
\begin{align*}
E|M_n/n-M|\leq & E|M^{\geq -m}_{n}/n-M^{\geq -m}|+|M^{\geq -m}-M|\notag\\
&+E\left((-m-X_0)\mathbb{1}[X_0\leq -m]\right).
\end{align*}
By \eqref{eq: truncated cvg}, for $m>0$, we have that
\[\limsup_{n\to\infty}E|M_n/n-M|\leq |M^{\geq -m}-M|+E\left((-m-X_0)\mathbb{1}[X_0\leq -m]\right).\]
Note that
\[\lim_{m\to\infty}|M^{\geq -m}-M|=0\]
and that
\[\lim_{m\to\infty}E\left((-m-X_0)\mathbb{1}[X_0\leq -m]\right)=0.\]
Hence, $\lim_{n\to\infty}E|M_n-M|=0$, i.e. $M_n\overset{n\to\infty}{\to}M$ in $L^{1}$.

Next, we prove the almost sure convergence: Denote by $Y$ a binomial random variable with parameter $(n,P(X_0\leq -m))$. Then, we have that
\[E\prod_{j=0}^{k-1}(Y-j)=\prod_{j=0}^{k-1}(n-j)\times P(X_0\leq -m)^k.\]
Hence, by Lemma~\ref{lem: key lemma}, we see that
\[E\prod_{j=0}^{k-1}(N_n(m)-j)\leq E\prod_{j=0}^{k-1}(Y-j).\]
By \cite[Eq. (4.1.3)]{RomanMR741185}, we have that
\[x^n=\sum_{k=0}^{n}S(n,k)x(x-1)\cdots (x-k+1),\]
where $S(n,k)\geq 0$ is the Stirling number of the second type. Hence, we have that
\[E(N_n(m)^k)\leq E(Y^k).\]
Therefore, for $t\geq 0$, we have that
\begin{multline*}
E\exp(tN_n(m))=\sum_{k=0}^{\infty}\frac{t^k}{k!}E(N_n(m)^k)\leq \sum_{k=0}^{\infty}\frac{t^k}{k!}E(Y^k)\\
=E\exp(tY)=((e^t-1)P(X_0\leq -m)+1)^n.
\end{multline*}
By Markov's inequality, we obtain that
\begin{multline*}
P(N_n(m)\geq 2P(X_0\leq -m)n)\\
\leq \exp(-2tP(X_0\leq -m)n)((e^t-1)P(X_0\leq -m)+1)^n.
\end{multline*}
Take $t=\ln\left(\frac{2(1-P(X_0\leq -m))}{1-2P(X_0\leq -m)}\right)$. Then, for sufficiently large $m$, we get that
\begin{equation}\label{eq: exponential upper bound for N_n(m)}
P(N_n(m)\geq 2P(X_0\leq -m)n)\leq \exp(-c(m)n),
\end{equation}
where $c(m)=2P(X_0\leq -m)\ln\left(\frac{2(1-P(X_0\leq -m))}{1-2P(X_0\leq -m)}\right)-\ln\left(\frac{1-P(X_0\leq -m)}{1-2P(X_0\leq -m)}\right)>0$. Conditionally on $N_n(m)=\ell$, $\sum_{v\in\pi^{\geq -m}_{n}}(-m-X_v)\mathbb{1}[X_v\leq -m]$ has the same distribution as $\xi_1+\xi_2+\cdots+\xi_{\ell}$, where $\xi_1,\xi_2,\ldots,\xi_{\ell}$ are independent with the same distribution $G(\mathrm{d}x)=P(-m-X_0\in\mathrm{d}x|X_0\leq -m)$. Also, we have that $E\xi_1=E\left((-m-X_0)\mathbb{1}[X_0\leq -m]\right)/P(X_0\leq -m)$. Note that
\begin{align}\label{eq: fourth moment}
E(\sum_{j=1}^{\ell}(\xi_j-E\xi_j))^4&=\ell E(\xi_1-E\xi_1)^4+6\ell(\ell-1)(E(\xi_1-E\xi_1)^2)^2\notag\\
&\leq 8\ell^2E(\xi_1-E\xi_1)^4\notag\\
&\leq 8\ell^2E\xi_1^4,
\end{align}
where the last inequality holds since $\xi_1\geq 0$. Hence, for any $\ell\leq 2P(X_0\leq -m)n$ and any $\varepsilon\geq 4E\left((-X_0-m)\mathbb{1}[X_0\leq -m]\right)$, we have that
\begin{align}\label{eq: Borel proba upper bound}
P(\sum_{j=1}^{\ell}\xi_j\geq\varepsilon n) &= P(\sum_{j=1}^{\ell}\xi_j-E\xi_j\geq\varepsilon n-\ell E\xi_1)\notag\\
& \leq P(\sum_{j=1}^{\ell}\xi_j-E\xi_j\geq \frac{\varepsilon}{2}n)\notag\\
& \overset{\text{Markov's inequality}}{\leq} \frac{16}{\varepsilon^4n^4}E(\sum_{j=1}^{\ell}\xi_j-E\xi_j)^4\notag\\
& \overset{\text{by \eqref{eq: fourth moment}}}{\leq }\frac{128\ell^2}{\varepsilon^4n^4}E\xi_1^4\notag\\
& \leq a(m,\varepsilon)/n^2.
\end{align}
Hence, for $\varepsilon=4E(-X_0-m)\mathbb{1}[X_0\leq -m]$ and $\ell=\lfloor 2P(X_0\leq -m)n\rfloor$, since $\xi_j\geq 0$, we have that
\begin{multline*}
P\left(\sum_{v\in\pi^{\geq -m}_{n}}(-m-X_v)\mathbb{1}[X_v\leq -m]/n\geq \varepsilon\right) \\
\leq P(N_n(m)\geq 2P(X_0\leq -m)n)+P(\sum_{j=1}^{\ell}\xi_j\geq\varepsilon n)\\
\overset{\eqref{eq: exponential upper bound for N_n(m)},\eqref{eq: Borel proba upper bound}}{\leq} \exp(-c(m)n)+a(m,\varepsilon)/n^2.
\end{multline*}
Since $\sum_{n\geq 1}\exp(-c(m)n)+a(m,\varepsilon)/n^2<+\infty$, by Borel-Cantelli lemma, the following inequality holds with probability one:
\[\limsup_{n\to\infty}\sum_{v\in\pi^{\geq -m}_{n}}(-m-X_v)\mathbb{1}[X_v\leq -m]/n\leq \varepsilon=4E(-X_0-m)\mathbb{1}[X_0\leq -m].\]
Together with the fact that $\lim_{n\to\infty}M^{\geq -m}_{n}/n=M^{\geq -m}$ a.s., by \eqref{eq: upper bound of |M_n-M|}, for any $m>0$, we have that
\begin{equation*}
\limsup_{n\to\infty}|M_n/n-M|\leq |M^{\geq -m}-M|+4E(-X_0-m)\mathbb{1}[X_0\leq -m]\overset{m\to\infty}{\to}0.
\end{equation*}
Consequently, $M_n/n\overset{n\to\infty}{\to}M$ almost surely.

Finally, we give the proof of Lemma~\ref{lem: key lemma}.
\begin{proof}[Proof of Lemma~\ref{lem: key lemma}]
 We only present the detailed proofs for $k=1$ and $k=2$. The general case could be proved similarly. For $k=1$, we have that
\begin{align*}
 E(N_n(m))&=\sum_{v\in\mathbb{Z}^d}P(v\in\pi^{\geq -m}_{n},X_{v}\leq -m)\\
 &=\sum_{v\in\mathbb{Z}^d}E(E(\mathbb{1}[v\in\pi^{\geq -m}_{n}]\mathbb{1}[X_v\leq -m]|\mathcal{F}_v)),
\end{align*}
where $\mathcal{F}_v$ is the sigma-field generated by $\{X_u:u\in\mathbb{Z}^d,u\neq v\}$. For fixed values of $\{X_u:u\in\mathbb{Z}^d,u\neq v\}$, $\mathbb{1}[v\in\pi^{\geq -m}_{n}]$ is an increasing function of $X_v$, and $\mathbb{1}[X_v\leq -m]$ is a decreasing function of $X_v$. For a increasing function $f$, a decreasing function $g$ and a probability measure $\mu$, we have that
\[\int f(x)g(x)\,\mu(\mathrm{d}x)\leq\int f(x)\,\mu(\mathrm{d}x)\int g(x)\,\mu(\mathrm{d}x).\]
This inequality is the continuous version of Chebyshev's sum inequality, see e.g. \cite[Section~7.1]{PJPTMR1162312}. By using this inequality, we obtain that
\[E(\mathbb{1}[v\in\pi^{\geq -m}_{n}]\mathbb{1}[X_v\leq -m]|\mathcal{F}_v)\leq E(\mathbb{1}[v\in\pi^{\geq -m}_{n}]|\mathcal{F}_v)E(\mathbb{1}[X_v\leq -m]|\mathcal{F}_v).\]
By independence between $X_v$ and $\mathcal{F}_v$, we have that
\[E(\mathbb{1}[X_v\leq -m]|\mathcal{F}_v)=P(X_v\leq -m)=P(X_0\leq -m).\]
Finally, we have that
\begin{align*}
 E(N_n(m))&\leq \sum_{v\in\mathbb{Z}^d}P(X_0\leq -m)E(E(\mathbb{1}[v\in\pi^{\geq -m}_{n}]|\mathcal{F}_v))\\
 &=\sum_{v\in\mathbb{Z}^d}P(X_0\leq -m)E(\mathbb{1}[v\in\pi^{\geq -m}_{n}])\\
 &=P(X_0\leq -m)E(\sum_{v\in\mathbb{Z}^d}\mathbb{1}[v\in\pi^{\geq -m}_{n}])\\
 &=nP(X_0\leq -m).
\end{align*}
For $k=2$, we have that
\begin{align*}
 EN_n(m)&(N_n(m)-1)\\
 &=\sum_{v,w\in\mathbb{Z}^d:v\neq w}P(v\in\pi^{\geq -m}_{n},X_{v}\leq -m,w\in\pi^{\geq -m}_{n},X_{w}\leq -m)\\
 &=\sum_{v\in\mathbb{Z}^d}E(E(\mathbb{1}[v\in\pi^{\geq -m}_{n},w\in\pi^{\geq -m}_{n}]\mathbb{1}[X_v\leq -m,X_w\leq -m]|\mathcal{F}_{v,w})),
\end{align*}
where $\mathcal{F}_{v,w}$ is the sigma-field generated by $\{X_u:u\in\mathbb{Z}^d,u\neq v,u\neq w\}$. By independence of $\mathcal{F}_{v,w}$ and $(X_v,X_w)$, given $\mathcal{F}_{v,w}$, the conditional distribution of $(X_v,X_w)$ is a product measure. Hence, FKG inequalities hold. (See \cite{FortuinKasteleynGinibreMR309498} for FKG inequalities for a distributive lattice.) Note that $\mathbb{1}[v\in\pi^{\geq -m}_{n},w\in\pi^{\geq -m}_{n}]$ is a non-decreasing function of $(X_v,X_w)$ and $\mathbb{1}[X_v\leq -m,X_w\leq -m]$ is a non-increasing function of $(X_v,X_w)$. Therefore, by FKG inequalities, we have that
\begin{multline*}
E(\mathbb{1}[v\in\pi^{\geq -m}_{n},w\in\pi^{\geq -m}_{n}]\mathbb{1}[X_v\leq -m,X_w\leq -m]|\mathcal{F}_{v,w})\\
\leq E(\mathbb{1}[v\in\pi^{\geq -m}_{n},w\in\pi^{\geq -m}_{n}]|\mathcal{F}_{v,w})E(\mathbb{1}[X_v\leq -m,X_w\leq -m]|\mathcal{F}_{v,w}).
\end{multline*}
Note that $E(\mathbb{1}[X_v\leq -m,X_w\leq -m]|\mathcal{F}_{v,w})=P(X_0\leq -m)^2$. Hence, we have that
\begin{align*}
 EN_n(m)&(N_n(m)-1)\\
 &\leq \sum_{v,w\in\mathbb{Z}^d:v\neq w}P(X_0\leq -m)^2E(E(\mathbb{1}[v\in\pi^{\geq -m}_{n},w\in\pi^{\geq -m}_{n}]|\mathcal{F}_{v,w}))\\
 &=\sum_{v,w\in\mathbb{Z}^d:v\neq w}P(X_0\leq -m)^2E(\mathbb{1}[v\in\pi^{\geq -m}_{n},w\in\pi^{\geq -m}_{n}])\\
 &=P(X_0\leq -m)^2E\left(\sum_{v,w\in\mathbb{Z}^d:v\neq w}\mathbb{1}[v\in\pi^{\geq -m}_{n},w\in\pi^{\geq -m}_{n}]\right)\\
 &= n(n-1)P(X_0\leq -m)^2.
\end{align*}
For general $k\geq 1$, we have that
\[E\prod_{j=0}^{k-1}(N_n(m)-j)=\sum_{\text{different }v_1,\ldots,v_k\in\mathbb{Z}^d}P(\cap_{j=1}^{k}\{v_j\in\pi^{\geq -m}_{n},X_{v_{j}}\leq -m\}).\]
The rest of the proof is similar to the case $k=2$.
\end{proof}

\section{Open problems}\label{sect: open problems}

We believe that the condition $E((X_0^{-})^4)<+\infty$ is not necessary for the almost sure convergence in Theorem~\ref{thm: linear growth}. It appears for purely technical reasons. Indeed, we need to prove that $\sum_{v\in\pi^{\geq -m}_{n}}(-m-X_v)\mathbb{1}[X_v\leq -m]/n$ is small for sufficiently large $m$ and $n$. However, if the weaker condition $E(X_0^{-})<+\infty$ is violated, by \eqref{eq: upper bound of expectation of increament along pimn}, the expectation of $\sum_{v\in\pi^{\geq -m}_{n}}(-m-X_v)\mathbb{1}[X_v\leq -m]/n$ is infinite, which is an obstruction.

We wish to remove this condition and to prove the almost sure convergence for weights with arbitrary negative tails in the future work, i.e. to prove the following conjecture.

\begin{conjecture}\label{conj: one}
Let $\{X_{v}:v\in\mathbb{Z}^d\}$ be i.i.d. random variables. Assume that there exists $\alpha>0$ such that
\[E((X_0^{+})^d(\log^{+}X_0^{+})^{d+\alpha})<+\infty.\]
Then, there exists a constant $M\in(-\infty,+\infty)$ such that $M_n/n\overset{n\to\infty}{\to}M$ almost surely.
\end{conjecture}

As we discussed, the absence of moment condition on the negative part $X_0^{-}$ causes technical issues. Why do we still believe that the same results should hold for weights $X_v$ with arbitrary negative tails? One reason is that we are looking at the greedy lattice paths with maximal weights. Besides, the vertices with very negative weights are well separated and its complement is very close to $\mathbb{Z}^d$ and percolates. Hence, it is easy for the optimal path to avoid the vertices with very negative weights. Therefore, those vertices with very negative weights have rather small influence on the weight $M_n$ of the optimal path of length $n$. Consequently, Conjecture~\ref{conj: one} may hold for $X_0$ with arbitrary negative tail. By the way, we would like to draw the attention to similar results on greedy lattice animals without the positive constraints, see \cite{DemboGandolfiKestenMR1825148}. As the almost sure convergence holds for greedy lattice animals with no constraints on the negative tails of $X_v$, by similarity between these two models, we feel that it is possible that Conjecture~\ref{conj: one} holds for the model of greedy lattice paths.

To solve the conjecture, it is natural to ask whether the argument in \cite{DemboGandolfiKestenMR1825148} for lattice animals still works for lattice paths. However, in \cite{DemboGandolfiKestenMR1825148}, the authors pointed out that their argument does not work for greedy lattice paths. So, we need new ideas to solve the problem of greedy lattice paths.

For the $L^{1}$ convergence part, certain integrability condition of $X_0^{-}$ is required. For instance, to ensure that $EM_1>-\infty$, we need at least
\begin{equation}\label{eq: cond for EM_1 > -infinity}
\int_{0}^{+\infty}P(X_0< -t)^{2d}\,\mathrm{d}t<+\infty.
\end{equation}
It turns out that $EM_n>-\infty$ under the assumption \eqref{eq: cond for EM_1 > -infinity}. Indeed, there exists $2d$ disjoint self-avoiding paths $\Gamma_1,\Gamma_2,\ldots,\Gamma_{2d}$ of length $n$ starting from the origin. Then, we have that
\[M_n\geq \max(S(\Gamma_1),S(\Gamma_2),\ldots,S(\Gamma_{2d})).\]
Therefore, for $t\geq 0$, we have that
\begin{align*}
 P(-M_n>t)&\leq P(-\max(S(\Gamma_1),S(\Gamma_2),\ldots,S(\Gamma_n))>t)\\
 &=P(S(\Gamma_j)<-t,\forall j=1,2,\ldots,2d)\\
 &=\prod_{j=1}^{2d}P(S(\Gamma_j)<-t).
\end{align*}
Note that $P(S(\Gamma_j)<-t)\leq P(\exists v\in S(\Gamma_j),X_v<-t/n)\leq nP(X_0<-t/n)$. Hence, we obtain that
\[P(-M_n>t)\leq n^{2d}P(X_0<-t/n)^{2d}.\]
Therefore, we have that
\begin{align*}
EM_n&=E M_n^{+}-E M_n^{-}\\
&=E M_n^{+}-\int_{0}^{+\infty}P(M_n^{-}>t)\,\mathrm{d}t\\
&=E M_n^{+}-\int_{0}^{+\infty}P(-M_n>t)\,\mathrm{d}t\\
&\geq E M_n^{+}-\int_{0}^{+\infty}n^{2d}P(X_0<-t/n)^{2d}\,\mathrm{d}t\\
&\geq E M_n^{+}-n^{2d+1}\int_{0}^{+\infty}P(X_0<-t)^{2d}\,\mathrm{d}t>-\infty.
\end{align*}
Finally, we propose the following conjecture.
\begin{conjecture}
Let $\{X_{v}:v\in\mathbb{Z}^d\}$ be i.i.d. random variables. Assume that there exists $\alpha>0$ such that
\[E(X_0^{+})^d(\log^{+}X_0^{+})^{d+\alpha})<+\infty\]
and that
\[\int_{0}^{+\infty}P(X_0< -t)^{2d}\,\mathrm{d}t<+\infty.\]
Then, there exists a constant $M\in(-\infty,+\infty)$ such that $E|M_n/n-M|\overset{n\to\infty}{\to}0$.
\end{conjecture}

\bibliographystyle{alpha}

\begin{thebibliography}{DGK01}

\bibitem[ADH17]{AuffingerDamronHansonMR3729447}
Antonio Auffinger, Michael Damron, and Jack Hanson.
\newblock {\em 50 years of first-passage percolation}, volume~68 of {\em
  University Lecture Series}.
\newblock American Mathematical Society, Providence, RI, 2017.

\bibitem[DGK01]{DemboGandolfiKestenMR1825148}
Amir Dembo, Alberto Gandolfi, and Harry Kesten.
\newblock Greedy lattice animals: negative values and unconstrained maxima.
\newblock {\em Ann. Probab.}, 29(1):205--241, 2001.

\bibitem[FKG71]{FortuinKasteleynGinibreMR309498}
C.~M. Fortuin, P.~W. Kasteleyn, and J.~Ginibre.
\newblock Correlation inequalities on some partially ordered sets.
\newblock {\em Comm. Math. Phys.}, 22:89--103, 1971.

\bibitem[GK94]{GandolfiKestenMR1258174}
Alberto Gandolfi and Harry Kesten.
\newblock Greedy lattice animals. {II}. {L}inear growth.
\newblock {\em Ann. Appl. Probab.}, 4(1):76--107, 1994.

\bibitem[PPT92]{PJPTMR1162312}
Josip~E. Pe\v{c}ari\'{c}, Frank Proschan, and Y.~L. Tong.
\newblock {\em Convex functions, partial orderings, and statistical
  applications}, volume 187 of {\em Mathematics in Science and Engineering}.
\newblock Academic Press, Inc., Boston, MA, 1992.

\bibitem[Rom84]{RomanMR741185}
Steven Roman.
\newblock {\em The umbral calculus}, volume 111 of {\em Pure and Applied
  Mathematics}.
\newblock Academic Press, Inc. [Harcourt Brace Jovanovich, Publishers], New
  York, 1984.

\bibitem[Shi15]{ShiMR3444654}
Zhan Shi.
\newblock {\em Branching random walks}, volume 2151 of {\em Lecture Notes in
  Mathematics}.
\newblock Springer, Cham, 2015.
\newblock Lecture notes from the 42nd Probability Summer School held in Saint
  Flour, 2012, \'{E}cole d'\'{E}t\'{e} de Probabilit\'{e}s de Saint-Flour.
  [Saint-Flour Probability Summer School].

\bibitem[SW78]{SmytheWiermanMR513421}
R.~T. Smythe and John~C. Wierman.
\newblock {\em First-passage percolation on the square lattice}, volume 671 of
  {\em Lecture Notes in Mathematics}.
\newblock Springer, Berlin, 1978.

\end{thebibliography}

\end{document}